\documentclass[12pt,oneside]{amsart}
\usepackage[foot]{amsaddr}
\usepackage[utf8]{inputenc}
\usepackage[T1]{fontenc}
\usepackage{lmodern}

\usepackage{amssymb}
\usepackage{geometry}
\usepackage{amsmath}
\usepackage{amsthm}
\usepackage{graphicx}
\usepackage{color}
\usepackage[normalem]{ulem}

\usepackage{verbatim}
\usepackage[hidelinks]{hyperref}

\newcommand{\IR}{\ensuremath{\mathbb{R}}}
\newcommand{\IN}{\ensuremath{\mathbb{N}}}

\newcommand{\B}{B_F}

\newcommand{\nall}{n^{\rm all}}
\newcommand{\nstd}{n^{\rm std}}
\newcommand{\nx}{n^{\rm x}}
\newcommand{\nstdlin}{n^{\text{\rm std-lin}}}

\newcommand{\eps}{\varepsilon}

\renewcommand{\rho}{\varrho}

\newcommand{\set}[1]{\left\{#1\right\}}

\DeclareMathOperator{\app}{APP}

\DeclareMathOperator{\rank}{rank}

\newtheorem{thm}{Theorem}
\newtheorem*{thm-intro}{Theorem}
\newtheorem{cor}[thm]{Corollary}

\theoremstyle{plain}
\newtheorem{lemma}[thm]{Lemma}

\theoremstyle{definition}

\newtheorem{rem}[thm]{Remark}

\title[Exponential tractability with function values]
{Exponential tractability of $L_2$-approximation \\ with function values}

\author{
David Krieg$^{1,2}$
\and
Paweł Siedlecki$^{3}$ 
\and 
Mario Ullrich$^{1}$
\and
Henryk Wo\'zniakowski$^{3,4}$
}

\address{$^1$ Institut f\"ur Analysis, 
Johannes Kepler Universit\"at, Linz, Austria}
\address{$^2$ Johann Radon Institute for Computational and Applied Mathematics (RICAM), Austrian Academy of Sciences, Linz, Austria.}
\address{$^3$ Faculty of Mathematics, Informatics and Mechanics, 
University of Warsaw, Warsaw, Poland} 
\address{$^4$ Department of Computer Science, 
Columbia University, New York, United States of America}
\email{
david.krieg@jku.at, 
psiedlecki@mimuw.edu.pl} 
\email{
mario.ullrich@jku.at, 
hwozniak@mimuw.edu.pl}

\date{\today}

\begin{document}
\sloppy

\begin{abstract}
We study the complexity of 
high-dimensional approximation 
in the $L_2$-norm when different classes of information are available;
we compare the power of function evaluations with the power of arbtirary continuous linear measurements.
Here, we discuss the situation when the number of 
linear measurements
required to achieve an error $\varepsilon \in (0,1)$ in dimension $d\in\IN$
depends only 
poly-logarithmically on $\eps^{-1}$. 
This corresponds to an exponential order of convergence of the approximation error, 
which often happens
in applications. 
However, it does not mean that the high-dimensional approximation problem is easy,
the main difficulty usually
lies within
the dependence on the dimension $d$.
We determine to which extent the required
amount of information changes if 
we allow only function evaluation instead of arbitrary linear information.
It turns out that in this case we only lose very little, and we can even 
restrict to linear algorithms. 
In particular, several notions of tractability hold simultaneously for both types 
of available information. 

\end{abstract}

\maketitle

\section{Exposition}

We want to approximate real- or complex-valued functions defined on some (nonempty) set $\mathcal{D}$, and belonging to a space $F$.  
We assume that $F$ is a separable Banach space of functions defined on $\mathcal{D}$, such that function evaluation $f\mapsto f(x)$ is continuous on $F$ for each $x\in\mathcal D$
and $F$ is continuously embedded in $L_2=L_2(\mathcal{D},\mu)$ for some 
measure $\mu$. 
Formally, the approximation problem is given as 
$$\app:F\to L_2, \quad \app(f):= f,$$
which might be understood as a continuous embedding into $L_2$. 
The class of all spaces $F$ satisfying the assumptions above will be 
denoted by $\mathcal{A}$. In particular, for each $F\in\mathcal{A}$ 
we have some associated nonempty set $\mathcal{D}$, 
measure $\mu$ on $\mathcal{D}$ and continuous embedding $\app$. 

We approximate $\app$ by using functionals from the class $\Lambda^\mathrm{std}$ consisting of all function evaluations, or from the class $\Lambda^\mathrm{all}=F^\ast$ of all continuous linear functionals. 

Below $\B$ denotes the closed unit ball in $F$. Let us define, for $n\in\IN$, the
\begin{itemize}
\item\emph{$n$-th linear sampling width} as
\begin{align*}
e_n(F,L_2) \,:=\, 
\inf_{\substack{x_1,\dots,x_n\in \mathcal{D}\\ \varphi_1,\dots,\varphi_n\in L_2}}\, 
\sup_{f\in \B}\, 
\Big\|f - \sum_{i=1}^n f(x_i)\, \varphi_i\Big\|_{L_2},
\end{align*}
\item\emph{$n$-th sampling width} as
\begin{align*}
g_n(F,L_2) \,:=\, 
\inf_{\substack{x_1,\dots,x_n\in \mathcal{D}\\ \phi\colon \IR^n \to L_2\\}}\, 
\sup_{f\in \B}\, 
\Big\|f - \phi(f(x_1),\dotsc,f(x_n)) \Big\|_{L_2},
\end{align*}
\item\emph{$n$-th linear width} as
\begin{align*}
a_n(F,L_2) \,:=\,
\inf_{\substack{T\colon L_2 \to L_2\\ \rank(T) \,\le\, n}}\, 
\sup_{f\in \B}\, 
\big\|f - Tf \big\|_{L_2},
\end{align*}
\item\emph{$n$-th Gelfand width} as
\begin{align*}
\label{def:cn}
c_n(F,L_2) \,:=\,
\inf_{\substack{\phi\colon \IR^n \to L_2\\ N\in (F^*)^n}}\, 
\sup_{f\in \B}\, 
\big\|f - \phi\circ N(f) \big\|_{L_2}.
\end{align*}
\end{itemize}
These quantities represent the \emph{minimal worst case errors} that 
can be achieved with linear or nonlinear algorithms using 
at most $n$ function values or linear measurements, respectively.

We also define the information-based complexity of the problem $\app$ 
for the classes $\Lambda^\mathrm{std}$ and $\Lambda^\mathrm{all}$, respectively, as the minimal number of evaluations from $\Lambda^\mathrm{std}$ or $\Lambda^\mathrm{all}$ necessary to obtain the absolute precision of approximation at most $\eps$, i.e., as 
\[
\nstd(\eps,F) \,:=\, \min\big\{n\colon g_n(F,L_2)\le\eps\big\}
\]
and
\[
\nall(\eps,F) \,:=\, \min\big\{n\colon c_n(F,L_2)\le\eps\big\}.
\]
Note that, since $g_n(F,L_2)\leq e_n(F,L_2)$, we have 
\[
\nstd(\eps,F) \, \leq \, \min\big\{n\colon e_n(F,L_2)\le\eps\big\}
\,\,=:\, \nstdlin(\eps,F),
\]
and all our upper bounds are proven
for $\nstdlin(\eps,F)$. 
There is a lot of literature on the size of these quantities 
for specific classes $F$. 
We refer to the monographs \cite{DTU18,NW08,NW10,NW12,Tem93,Tem18,W04} 
for more details and literature on the subject.

Here, we are specifically interested in the comparison of these quantities
for general classes $F$. 
That is, since $\nall(\eps,F)\le\nstd(\eps,F)$ is obvious for all $F\in\mathcal{A}$, 
we ask for an upper bound on $\nstd(\eps,F)$ based on knowledge of the function $\nall(\eps,F)$.
However, it is known that such a bound cannot hold without certain assumptions on 
$F$, see~\cite[Chapter~26]{NW12} and references therein, 
and even then, the involved ``constants'' depend in a non-trivial way on $F$. 
One approach to obtain qualitative statements on the relation of the complexities 
is to consider a whole sequence of spaces $(F_d)_{d\in\IN}$, 
where $d$ can be interpreted as the dimension of the underlying domain. 
We then assume a certain bound on $\nall(\eps,F_d)$, 
depending only on $\eps\in(0,1)$ and $d\in\IN$, 
and ask for an upper bound on $\nstd(\eps,F_d)$, 
hopefully not much worse than the bound on $\nall(\eps,F_d)$.

In the present paper, we allow 
arbitrary Banach spaces of functions $F_d$, 
but we assume that $\nall(\eps,F_d)$ depends only
poly-logarithmically on $\eps^{-1}$.
That is, we assume  
that there exist $A_d, B_d>0$ such that
\begin{equation}\label{eq:n}
\nall(\eps,F_d) \le A_d \left( 1 + \ln\eps^{-1}\right)^{B_d}
\qquad\text{for all } 0<\eps \le 1,
\end{equation}
and study how this translates into bounds on $\nstd(\eps,F_d)$. 
Note that the above bound \eqref{eq:n} on the complexity implies that 
\begin{equation}\label{eq:c}
c_n(F_d,L_2)\leq e\, \exp(-(n/A_d)^{1/B_d})
\qquad\text{for all } n \ge A_d,
\end{equation}
whereas \eqref{eq:c} implies that \eqref{eq:n} holds with $+1$ added on the right hand side.
The assumption \eqref{eq:n} is therefore equivalent to the existence of a (possibly non-linear)
algorithm based on arbitrary linear information that converges exponentially fast. 
We will show that in this case,
we do not lose much when we only 
allow linear algorithms and function evaluations as information.
One of our main results may be stated as follows.

\begin{thm-intro}[see Corollary~\ref{cor:main}]
Assume that  
$F_d\in\mathcal{A}$ for every $d\in\IN$ and 
\begin{align*}
\nall(\eps,F_d)\,\leq\, c\, d^q\, (1+\ln\eps^{-1})^p
\end{align*} 
for some $p,c>0$, $q\geq 0$, and all $\eps\in (0,1)$. Then 
\begin{align*}
\nstdlin(\eps,F_d)\, \leq\, C\, d^q\, (1+\ln d)^p\, (1+\ln\eps^{-1})^p
\end{align*} 
for all $\eps\in (0,1)$ and $d\in\IN$, and some $C>0$ 
that depends only on $c$, $p$ and $q$. 
\end{thm-intro}

This shows that every Banach space that is 
assumed to be \emph{approximable in high dimensions} (in the above sense) 
with an exponential order by some algorithm and information  
can practically already be treated with linear algorithms 
based on function values.
In particular, 
this improves upon Theorem~26.21 from \cite{NW12}
and solves Open Problem~128 therein. 
Let us add that we do not know if the additional 
$(1+\ln d)^p$ is necessary.

There are many appearances of the 
assumption~\eqref{eq:n} in the 
literature. 
Besides the detailed study of certain weighted Hilbert spaces of 
analytic functions \cite{DKPW14,IKPW16,IKPW16b,KW19,Z21}, 
it appears naturally in the context of approximation with 
(increasingly flat) Gaussian kernels~\cite{FHW12,HR10,KS20,SW18}, 
or in tensor product approximations~\cite{GO16,HK07}, 
or for certain smoothness spaces on complex spheres~\cite{AT21}. 
Moreover, it is a typical assumption for the construction 
of greedy bases~\cite{BCDDPW11,BMPPT12,Ha13}.
Let us also add that there is quite some study on the \emph{stability} 
of algorithms that can achieve an exponential convergence, 
see~\cite{AHS14,AHM14,APS18,PTK11} for details.

When 
it comes to the study of the \emph{tractability} 
of the problem, i.e., the precise dependence of the error on the dimension, 
especially when we only allow function evaluations, 
there is much less to cite 
and we are only aware of the 
Hilbert space references from above.
As an explicit example, 
let us mention the Gaussian space on $\IR^d$ with reproducing kernel
$K(x,y)=\exp(-\Vert x-y\Vert_2^d)$,
which satisfies a relation of the form \eqref{eq:n}
for $L_2$-approximation with respect to the Gaussian probability measure $\mu$,
see \cite{SW18}.
In the Hilbert case, 
there are 
some general results
which make the situation somewhat simpler.
For example, it is known that linear algorithms are always optimal
and one may work with the singular value decomposition of the embedding $\app$.
We refer to 
\cite[Chapter~4]{NW08} and~\cite[Chapter~26]{NW12}.

A bit more is known in the case of  
\emph{algebraic tractability}, i.e., 
when the complexity depends polynomially on $\eps^{-1}$ 
instead of $\ln\eps^{-1}$.
In addition to general 
Hilbert space results, see 
\cite[Chapter~26]{NW12},
and characterizations for weighted Korobov spaces, see~\cite{EP21},
there are also quite sharp results for 
the classical smoothness spaces 
$C^k(\Omega_d)$ 
of $k$-times differentiable functions on certain $d$-dimensional domains, 
possibly for $k=\infty$.
See~\cite{K19,NW09,X15} for details on approximation, 
and~\cite{HNU14,HNUW14,HNUW17,HPU19} 
for numerical integration in the same classes.
However, a comparison as proven here in the case of exponential convergence 
is not possible in this case, 
see the end of Section~\ref{sec:notions}.
In any case, 
it is open to determine the precise behavior 
of $\nstd(\eps,F_d)$ for most classical spaces,  
while $\nall(\eps,F_d)$ is more often known.

Our results are based on the following (special case of) Theorem~3 
from \cite{DKU22}, 
see also \cite{HKNU21,KU20,KU21,NSU20,Ull20}, 
which allows us to treat more general classes of functions.

\begin{thm}\label{thm:DKU}
For each $0<r<2$, 
there is a universal constant 
$b\in \IN$, 
depending only on $r$, such that 
the following holds. 
For all $F\in\mathcal{A}$ 
and
$n\ge 2$, we have
 \[
 e_{b n}(F, L_2) \,\le\, 
\left(\frac1n \sum_{k\geq n} a_k(F, L_2)^r \right)^{1/r}.
 \]
\end{thm}

\medskip

Additionally, we use the following fundamental result from~\cite{P78}, 
see also~ \cite{CW04,M90}.

\begin{thm}\label{thm:pietsch}
For all $F\in\mathcal{A}$ and $n\ge 1$, we have
\begin{equation*}\label{Pietsch}
 a_n(F, L_2) \,\le\, \left(1+\sqrt{n}\,\right) c_n(F,L_2).
\end{equation*} 
\end{thm}

\medskip

\begin{rem}\label{rem:non-con}
We would like to stress that the proof of Theorem~\ref{thm:DKU} in \cite{DKU22} is non-constructive, and we do not know 
how to explicitly construct evaluation points $x_1,\dotsc, x_{bn}\in \mathcal{D}$ 
together with some 
$\varphi_1,\dotsc,\varphi_{bn}\in L_2$ satisfying 
$$\sup_{f\in \B}\, \Big\|f - \sum_{i=1}^{bn} f(x_i)\, \varphi_i\Big\|_{L_2}\leq 
\left(\frac1n \sum_{k\geq n} a_k(F, L_2)^r \right)^{1/r}.$$ 
However, for problems with known operators $T$ achieving the infimum as in 
the definition of linear widths $a_n(F,L_2)$, 
we are able to specify algorithms utilizing i.i.d. sampling of evaluation 
points from some known distribution, 
and satisfying inequalities similar to the one above with high probability,  
see Theorem 8 of \cite{KU21}. 
\end{rem}


\section{Exponential tractability of approximation}
\label{sec:notions}

The notions of tractability are defined as follows. Let us fix, for every $d\in\IN$, some space $F_d\in\mathcal{A}$.  For each $F_d\in\mathcal{A}$ we have some associated set $\mathcal{D}_d$ equipped with a measure $\mu_d$, and a continuous embedding $\app_d:F_d\to L_2(\mathcal{D}_d,\mu_d)$.
The index $d\in\IN$ is an arbitrary parameter, but it usually stands for the dimension of the domain $\mathcal{D}_d$.
\emph{A multivariate approximation problem} is simply a sequence of embeddings 
$$\widetilde{\mathrm{APP}}=
\big(\mathrm{APP}_d:F_d\to L_2(\mathcal{D}_d,\mu_d)\big)_{d\in\IN}.$$ 

Moreover, tractability notions are defined relative to the considered class of information operations, i.e., we can consider tractability for 
$\Lambda^\mathrm{std}$ or $\Lambda^\mathrm{all}$. 
Therefore, for $\mathrm{x}\in\set{\mathrm{std},\mathrm{all}}$, we say that 
$\widetilde{\mathrm{APP}}$ is 
\begin{itemize}
\item \emph{exponentially strongly polynomially tractable} (EXP-SPT) for the class 
$\Lambda^\mathrm{x}$ if and only if  
$$\nx(\eps,F_d)\, \leq\, C\, (1+\ln\eps^{-1})^p$$
for some $C,p>0$ and for all $d\in\IN$ and $\eps\in (0,1)$, 
\item \emph{exponentially polynomially tractable} (EXP-PT) for the class 
$\Lambda^\mathrm{x}$ if and only if 
$$\nx(\eps,F_d)\, \leq\, C\, d^q\, (1+\ln\eps^{-1})^p$$
for some $C,p,q>0$ and for all $d\in\IN$ and $\eps\in (0,1)$,
\item \emph{exponentially quasi-polynomially tractable} (EXP-QPT) for the class 
$\Lambda^\mathrm{x}$ if and only if 
$$\nx(\eps,F_d)\, \leq\, C\, \exp(t(1+\ln d)(1+\ln(1+\ln\eps^{-1})))$$
for some $C,t>0$ and for all $d\in\IN$ and $\eps\in (0,1)$,  
\item \emph{exponentially uniformly weakly tractable} (EXP-UWT) for the class 
$\Lambda^\mathrm{x}$ if and only if for all 
$\alpha,\beta>0$  we have 
$$\lim_{d+\eps^{-1}\to\infty}\, \frac{\ln\nx(\eps,F_d)}{d^{\,\alpha}+(1+\ln\eps^{-1})^\beta}=0,$$ 
\item \emph{exponentially weakly tractable (EXP-WT)} for the class 
$\Lambda^\mathrm{x}$ if and only if 
$$\lim_{d+\eps^{-1}\to\infty}\, \frac{\ln\nx(\eps,F_d)}{d+(1+\ln\eps^{-1})}=0.$$
\end{itemize}

It is easy to see that we have the following logical relation between the tractability notions defined above 
$$\text{EXP-SPT}\implies \text{EXP-PT}\implies \text{EXP-QPT}\implies \text{EXP-UWT}\implies \text{EXP-WT}.$$

For a multivariate approximation problem we 
prove that \emph{exponential strong polynomial tractability} (EXP-SPT), 
\emph{exponential polynomial tractability} (EXP-PT), 
\emph{exponential uniform weak tractability} (EXP-UWT) 
and \emph{exponential weak tractability} (EXP-WT) for the class
$\Lambda^\mathrm{all}$ are each equivalent to the corresponding
tractability property for the class $\Lambda^\mathrm{std}$. 
Moreover, 
\emph{exponential quasi-polynomial tractability} (EXP-QPT) for 
$\Lambda^\mathrm{all}$ implies 
\emph{exponential uniform weak tractability} (EXP-UWT) for 
$\Lambda^\mathrm{std}$, i.e, the next tractability notion in the tractability hierarchy considered here. Whether the equivalence of 
\emph{exponential quasi-polynomial tractability} (EXP-QPT) for the classes $\Lambda^\mathrm{all}$ and $\Lambda^\mathrm{std}$ holds remains an open problem.

These equivalences are in sharp contrast to the results for 
algebraic tractability. 
See, e.g.,~\cite{HKNV21,NW16,V20} for examples 
where 
the problem is algebraically tractable for $\Lambda^\mathrm{all}$ but
the curse of dimensionality 
holds for $\Lambda^\mathrm{std}$.
In particular, \cite[Example~5]{NW16} shows that for
the tensor product $W_{2,d}^s$ of certain univariate periodic 
Sobolev spaces, $s>1/2$, 
we have QPT for $\Lambda^\mathrm{all}$, but the curse of dimensionality
for $\Lambda^\mathrm{std}$.

\section{Results} 

We now present our results. 
The first results are concerned with EXP-(S)PT and EXP-QPT. 
Both are direct corollaries of the following theorem.

\begin{thm}\label{thm:main1}
Assume that $F\in\mathcal{A}$ satisfies 
\begin{align*}
\nall(\eps,F) 
\,\le\, A\, \bigl( 1 + \ln\eps^{-1}\bigr)^B
\end{align*}
for some $B>0$ and $A\ge 1$ and all $\eps\in(0,1)$.  
Then    
\begin{align*}
\nstd(\eps,F) \,\le\, \nstdlin(\eps,F_d) 
\,\le\, C\, \bigl( 1 + \ln\eps^{-1}\bigr)^B
\end{align*}
for all $\eps\in(0,1)$,
where 
$$C= 3b\, A\, \left( \ln(36A)\,(1+B^3)\right)^B$$
and $b$ is the absolute constant from Theorem~\ref{thm:DKU} in the case $r=1$.
\end{thm}

\begin{proof}
Observe that the inequality 
$$\nall(\eps,F) \leq A \bigl( 1 + \ln\eps^{-1}\bigr)^B$$ 
implies that  
$$c_n(F,L_2)\leq e\, \exp(-(n/A)^{1/B}).$$
We obtain from Theorem~\ref{thm:pietsch}, 
and 
$1+n^{1/2}\le2n^{1/2}$, 
that
$$a_n(F,L_2) \leq 2\, e\, n^{1/2}\, \exp(-(n/A)^{1/B}).$$ 
Applying first
Lemma~\ref{series-leq-integral} and then
Lemma~\ref{integral-leq-exp} from the Appendix,
we deduce that 
\begin{align*}
\sum_{k\geq n}a_k(F,L_2) &\leq 2\,e\sum_{k\geq n}k^{1/2}\exp(-(k/A)^{1/B})\\
&\leq 6\, A^{1/B}\, B\, \max(3B/2,1) \, (n-1)^{3/2-1/B}\, \exp(-((n-1)/A)^{1/B})
\end{align*} 
for all $n\geq n_0(A,B) := A\,\max(3B/2,1)^B+1$. 
In particular, $(a_n(F, L_2))\in\ell_1$. 
It follows from Theorem~\ref{thm:DKU} that there 
exists an absolute constant $b\in\IN$
such that 
\begin{align*}
e_{bn}(F,L_2) &\leq\, 
n^{-1}\, \sum_{k\geq n}a_k(F,L_2) \\
&\leq\, 6\, 
A^{1/B}\, B\, \max(3B/2,1) \, (n-1)^{1/2-1/B}\, \exp(-((n-1)/A)^{1/B})
\end{align*} 
for all $n\geq n_0(A,B)$.

In the case $B\le 2$ we have
$(n-1)^{1/2-1/B}\le A^{1/2-1/B}$, and thus
\[
 e_{bn}(F,L_2) \,\le\, 36\, A^{1/2} \, \exp(-((n-1)/A)^{1/B}).
\]
If $B>2$, then 
Lemma~\ref{nuexp-leq-exp} with $u=1/2-1/B$ yields for any 
$\delta>0$ 
that
\begin{align*}
e_{bn}(F,L_2)\leq 9 \, A^{1/2}\, B^2\, 
\delta^{1-B/2}\, \exp\left(((B/2-1)\delta-1)((n-1)/A)^{1/B}\right)
\end{align*}
and taking 
$\delta=2/B$ yields 
\begin{align*}
e_{bn}(F_d,L_2) \,\leq\, 
36\, A^{1/2} (B/2)^{B/2+1}
\exp\left(-\frac{2}{B} \left(\frac{n-1}{A}\right)^{1/B}\right).
\end{align*}
If we put $B_0 := \max\{B/2,1\}$, we have for all $B>0$
and $n\geq n_0(A,B)$ the bound
\[
e_{bn}(F_d,L_2) \,\leq\, 36\, A^{1/2} B_0^{B_0+1}
\exp\left(-\frac{1}{B_0} \left(\frac{n-1}{A}\right)^{1/B}\right)
\]
which is smaller than $\varepsilon$ if
\[
 n \,\ge\,  A\, B_0^B
\left(\ln \left(36\, A^{1/2} B_0^{B_0+1} \eps^{-1} \right)\right)^B +1.
\]
Thus 
\begin{align*}
\nstd(\eps, F_d) &\leq\, b\, \max \left\{ 
A\, B_0^B
\left(\ln \left( 36\, A^{1/2} B_0^{B_0+1}  \eps^{-1} \right)\right)^B+2, 
\ n_0(A,B)
\right\}\\
&\le\, 3b\, \max \left\{ 
A\, B_0^B
R^B\left(1 + \ln(\eps^{-1}) \right)^B, 
\ A\, (3B_0)^B
\right\} \\
& \le\, 3b\, A\, B_0^B R^B\left(1 + \ln(\eps^{-1}) \right)^B 
\end{align*}
with 
$$
R\,:=\,\ln (36) + \frac{\ln A}{2} + (B_0+1) \ln B_0 
\,\le\, \ln(36A)\, B_0^2
$$
which gives the desired estimate.\\
\end{proof}

\begin{cor}\label{cor:main}
Assume that  
$F_d\in\mathcal{A}$ for every $d\in\IN$ and 
\begin{align*}
\nall(\eps,F_d)\,\leq\, c\, d^q\, (1+\ln\eps^{-1})^p
\end{align*} 
for some $p,c>0$, $q\geq 0$, and all $\eps\in (0,1)$. Then 
\begin{align*}
\nstd(\eps,F_d)\, \leq\, \nstdlin(\eps,F_d)\, \leq\, C\, d^q\, (1+\ln d)^p\, (1+\ln\eps^{-1})^p
\end{align*} 
for all $\eps\in (0,1)$ and $d\in\IN$, and some $C>0$ 
that depends only on $c$, $p$ and $q$. 

In particular, if $\widetilde{\app}$ is exponentially (strongly) 
polynomially tractable for the class $\Lambda^\mathrm{all}$ 
then it is exponentially (strongly) polynomially tractable for 
$\Lambda^\mathrm{std}$.
\end{cor}

\begin{proof}
We use Theorem~\ref{thm:main1} with $A=c\, d^q +1$ and $B=p$. \\
\end{proof}

We now turn to the assumption that 
$\widetilde{\mathrm{APP}}$ is 
exponentially quasi-polynomially tractable 
for the class~$\Lambda^\mathrm{all}$. 
This is the only case where we do not know if it implies the same 
property for~$\Lambda^\mathrm{std}$.

For convenience, let us write $\ln_+(x):=1+\ln(x)$. 

\begin{cor}\label{exp-qpt-to-uwt}
Assume that  
$F_d\in\mathcal{A}$ for every $d\in\IN$ and
\begin{align*}
\nall(\eps,F_d)\,\leq\, c\, \exp\left( t\cdot \ln_+d\cdot \ln_+\ln_+\eps^{-1} \right)
\end{align*} 
for some $c,t>0$ and all $\eps\in (0,1)$. Then 
\begin{align*}
\nstd(\eps,F_d)\, \leq\, c\,  
\exp\left( t\cdot \ln_+d\cdot 
\Bigl(
\ln_+\ln_+\eps^{-1} \,+\, 4\ln\bigl(t\, \ln_+d\bigr) \,+\, C
\Bigr)
\right)
\end{align*} 
for all $\eps\in (0,1)$ and 
$d > (e+\frac1c)^{1/t}e^{-1}$, 
and some $C>0$ 
that depends only on $c$. 

In particular, 
if $\widetilde{\mathrm{APP}}$ is 
exponentially quasi-polynomially tractable 
for the class~$\Lambda^\mathrm{all}$, 
then it is exponentially uniformly weakly tractable for the 
class $\Lambda^{std}$.
\end{cor}

\begin{proof}
Note that 
\begin{align*}
c\, \exp\left( t\, (1+\ln d)\, (1+\ln(1+\ln\eps^{-1})) \right)\, =\, 
c\, e^t \, d^{\,t} \, \left(1+\ln\eps^{-1}\right)^{t(1+\ln d)}.
\end{align*} 
Hence, we can apply Theorem~\ref{thm:main1} 
with $A=c\, e^t \, d^{\,t}$ and 
$B=t\, \ln_+ d$, i.e., $A=c\, e^B$.
Note that $A,B\ge1$ for 
$d > (e+\frac1c)^{1/t}e^{-1}$.
We obtain that there exists an absolute constant 
$b>0$ 
such that 
\begin{align*}
\nstd(\eps,F) 
\,\le\, C\, \bigl( 1 + \ln\eps^{-1}\bigr)^B
\end{align*}
for all $\eps\in(0,1)$,
with 
\[
\begin{split}
C\,&=\, 3b\, A\, \left( \ln(36A)\,(1+B^3)\right)^B
\le\, 3b\, A\, (2B)^{3B}\left( \ln(36c)+B\right)^B \\ 
&\le\, c\, \exp\left( B \Bigl(c'+4\ln(B)\Bigr) \right)
\end{split}
\]
where $c'>0$ only depends on $c$. 
This proves the bound.

Now, since $\ln\nstd(\eps,F_d)$ depends only logarithmically on $d$ 
and double-logarithmically on $\eps^{-1}$, we obtain 
$$\lim_{d+\eps^{-1}\to\infty}\, \frac{\ln\nstd(\eps,F_d)}{d^{\,\alpha}+(1+\ln\eps^{-1})^\beta}=0$$ 
for all $\alpha,\beta>0$, i.e., 
$\widetilde{\mathrm{APP}}$ is exponentially uniformly weakly tractable 
for the class~$\Lambda^\mathrm{std}$.\\
\end{proof}

We finally discuss EXP-UWT and EXP-WT.

\begin{thm}\label{thm:main2}
Assume that  
$F_d\in\mathcal{A}$ for every $d\in\IN$. If the problem $\widetilde{\mathrm{APP}}$ 
is exponentially (uniformly) weakly tractable for the class 
$\Lambda^\mathrm{all}$, then it is exponentially (uniformly) weakly tractable for the class~$\Lambda^\mathrm{std}$.
\end{thm}

\begin{proof}
Assume that there are $0<\alpha,\beta\le 1$ such that
\begin{equation*}
\lim_{d+\eps^{-1}\to\infty}\, \frac{\ln\nall(\eps,F_d)}{d^\alpha+(1+\ln\eps^{-1})^\beta}=0.
\end{equation*}
It is enough to show that
\begin{equation}\label{eq:weak-std}
\lim_{d+\eps^{-1}\to\infty}\, \frac{\ln\nstd(\eps,F_d)}{d^\alpha+(1+\ln\eps^{-1})^\beta}=0.
\end{equation}
By assumption, for every $0<h\le 1/16$, there is some $v_0\in\IN$
such that 
$$0 \leq \frac{\ln\nall(\eps,F_d)}{d^\alpha+(1+\ln\eps^{-1})^\beta} \leq h$$ 
for all $\eps\in(0,1)$ and $d\in\IN$ with $d^\alpha+(1+\ln\eps^{-1})^\beta\ge v_0$.
It follows that 
$$c_n(F_d,L_2)\leq e \exp\left(-\left(\frac{\ln n}{h}-d^\alpha\right)^{1/\beta}\right)$$
for all $n\ge \exp(h v_0)$. 
From Theorem~\ref{thm:pietsch}, we get 
\[
a_n(F_d,L_2)\leq 2e \exp\left(-\left(\left(\frac{\ln n}{h}-d^\alpha\right)^{1/\beta}-\frac12 \ln n\right)\right).
\]
For all $n\ge \exp(2hd^\alpha)$ and $h\le 1/16$, we have
\[
 \left(\frac{\ln n}{h}-d^\alpha\right)^{1/\beta}-\frac12 \ln n
 \ge \frac12
 \left(\frac{\ln n}{h}-d^\alpha\right)^{1/\beta}+\frac{1}{8} \frac{\ln n}{h}
\]
and hence we have for all $n\ge \max\{\exp(hv_0),\exp(2hd^\alpha)\}$ that
\[
a_n(F_d,L_2)\leq 2e \exp\left(-\frac12 \left(\frac{\ln n}{h}-d^\alpha\right)^{1/\beta}\right) \cdot n^{-1/(8h)}.
\]
It follows from Theorem~\ref{thm:DKU} that
for some absolute constant $b\in\IN$ 
and all $n\ge \max\{\exp(hv_0),\exp(2hd^\alpha)\}$, we have
\begin{align*}
e_{bn}(F_d, L_2) &\leq \frac{1}{n} \sum_{k\geq n}a_k(F_d,L_2)
\leq \sum_{k\geq n}a_k(F_d,L_2)\\
&\leq 2e \exp\left(-\frac12 \left(\frac{\ln n}{h}-d^\alpha\right)^{1/\beta}\right) 
\sum_{k\ge n} k^{-1/(8h)} \\
&\le 2e \exp\left(-\frac12 \left(\frac{\ln n}{h}-d^\alpha\right)^{1/\beta}\right),
\end{align*}
where we again used that $h\le 1/16$.
It follows that 
\begin{align*}
 \nstd(\eps,F_d)\leq D \exp\left( 4h\left((1+\ln\eps^{-1})^\beta + d^\alpha\right) \right)
\end{align*} 
for some absolute constant $D>0$ and all $\eps\in(0,1)$ and $d\in\IN$ such that $d+\eps^{-1}$ is sufficiently large. This implies 
$$0 \leq \lim_{d+\eps^{-1}\to\infty}\, \frac{\ln\nstd(\eps,F_d)}{d^\alpha+(1+\ln\eps^{-1})^\beta} 
\leq 4 h.$$
Since $h\in(0,1/16)$ can be chosen arbitrarily close to $0$, 
we obtain \eqref{eq:weak-std}.\\
This allows us to conclude our statement. Indeed, for uniform weak tractability we take arbitrary $\alpha$ and 
$\beta$ from $(0,1)$, and for weak tractability we take $\alpha=\beta=1$.
\end{proof}

\section{Appendix: technical lemmas}

The following lemmas are used in the proofs of our results.

\begin{lemma}\label{series-leq-integral}
Let $A$ and $B$ be arbitrary positive real numbers. 
For $n\ge A(B/2)^B$ we have the following inequality
\begin{align*}
\sum_{k\geq n+1}k^{1/2}\exp(-(k/A)^{1/B})\leq \int_n^\infty t^{1/2}\exp(-(t/A)^{1/B})dt.
\end{align*}
\end{lemma}
\begin{proof}
It is enough to show that the function $f:(0,\infty)\to \IR$ given by $f(t)=t^{1/2}\exp(-(t/A)^{1/B})$ is decreasing on  
$(A(B/2)^B,\infty)$. Indeed, for $t>A(B/2)^B$ we have 
\begin{align*}
f'(t)&=\left( \exp\left(\frac{1}{2}\ln(t)-(t/A)^{1/B}\right) \right)'=\\
&=\exp\left(\frac{1}{2}\ln(t)-(t/A)^{1/B}\right)\left( \frac{1}{2t}-(1/AB)(t/A)^{1/B-1} \right)<0.
\end{align*}
\end{proof}

\begin{lemma}\label{integral-leq-exp}
Let $A$ and $B$ be arbitrary positive real numbers. 
For every $n\ge A\,\max(3B/2,1)^B$ we have the following inequality
\begin{align*}
\int_n^\infty t^{1/2}\exp(-(t/A)^{1/B})dt \, \leq \,  
A^{1/B}\, B\, \max(3B/2,1) \, n^{3/2-1/B}\, \exp(-(n/A)^{1/B}).
\end{align*}
\end{lemma}
\begin{proof}
Using integration by substitution, with $u=(t/A)^{1/B}$, we obtain that 
$$
\int_n^\infty t^{1/2}\exp(-(t/A)^{1/B})dt \, = \, 
A^{3/2}\, B\, \Gamma(3B/2, (n/A)^{1/B})
$$
where, for $a\in\IR$ and $x>0$, $\Gamma(a,x)=\int_x^\infty v^{a-1}\exp(-v)dv$ is the incomplete gamma function.

It is known (see, e.g., Satz 4.4.3 in \cite{G79}) that for $a\geq 1$ and $x>a$ 
we have 
$$
\Gamma(a,x)\leq a\, x^{a-1}\, \exp(-x).
$$
If, on the other hand, $0 < a<1$ and $x>1$ then 
since $v^{a-1}\leq x^{a-1}$ for $v\geq x$ we have 
\begin{align*}
\Gamma(a,x)&=\int_x^\infty v^{a-1}\exp(-v)dv\leq x^{a-1}\int_x^\infty\exp(-v)dv=
x^{a-1}\exp(-x).
\end{align*}
Therefore, for every $a>0$ and $x>\max(a,1)$, the following bound holds
$$
\Gamma(a,x)\, \leq\, \max(a,1) \, x^{a-1}\, \exp(-x).
$$
Thus for $n>A\,\max(3B/2,1)^B$, and taking $a=3B/2$ and $x=(n/A)^{1/B}$, 
we have 
\begin{align*}
\int_n^\infty t^{1/2}\exp(-(t/A)^{1/B})dt \, \leq \,  
A^{1/B}\, B\, \max(3B/2,1) \, n^{3/2-1/B}\, \exp(-(n/A)^{1/B}).
\end{align*}
\end{proof}

\begin{lemma}\label{nuexp-leq-exp}
For every $A,B,n,\delta,u>0$ we have the following inequality 
\begin{align*}
n^u\, \exp\left(-(n/A)^{1/B}\right)\, 
\leq\, 
A^u\, \delta^{-uB}\, \exp\left((uB\delta-1)(n/A)^{1/B}\right).
\end{align*}
\end{lemma}
\begin{proof}
Let $x=\delta(n/A)^{1/B}$. Then $n^u=A^u\, \delta^{-uB}\, x^{uB}$. Using the fact that $\ln(x)\leq x$ for all $x>0$ we obtain that 
\begin{align*}
\ln(n^u) &=\ln(A^u\delta^{-uB}) + uB\ln(x)\leq\\ 
&\leq\,  \ln(A^u\delta^{-uB}) + 
uB x \, =\, \ln(A^u\delta^{-uB}) + 
uB\delta(n/A)^{1/B}
\end{align*}
 Hence, taking exponentials of both sides we derive that 
 \begin{align*}
 n^u\leq A^u\delta^{-uB}\exp(uB\delta(n/A)^{1/B})
\end{align*} 
and thus 
\begin{align*}
n^u\exp(-(n/A)^{1/B}) \leq A^u\delta^{-uB}\exp((uB\delta-1)(n/A)^{1/B})
\end{align*}
as claimed.
\end{proof}

\vskip 2pc

\noindent \textbf{Funding} David Krieg is supported by the Austrian Science Fund (FWF) Project F5506, which is part of the Special Research Program ``Quasi-Monte Carlo Methods: Theory and Applications''.

\vskip 1pc

\noindent \textbf{\large Declarations}

\vskip 1pc

\noindent \textbf{Conflict of interest} The authors declare no competing interests.

\bibliographystyle{plain}
\bibliography{bibIBC}

\end{document}